\documentclass[a4paper,11pt, reqno]{amsart}
\usepackage{amssymb,amsthm,amsmath}

\usepackage{enumerate}
\usepackage{ifthen}
\usepackage{graphicx} 
\usepackage[colorlinks,citecolor=blue,pagebackref,hypertexnames=false]{hyperref}

 \numberwithin{equation}{section}
\newtheorem{thm}{Theorem}[section]

\newtheorem{lem}[thm]{Lemma}

\newtheorem{defn}[thm]{Definition}

\allowdisplaybreaks[2]

\theoremstyle{definition}
\newtheorem{rmk}[thm]{Remark}

{\qed\bigskip}

\newcounter{alphabet}

\ifx\undefined\bysame
\newcommand{\bysame}{\leavevmode\hbox to3em{\hrulefill}\,}
\fi

\title[Estimates for the nonlinear viscoelastic damped wave equation on compact Lie  groups]
{Estimates for the nonlinear viscoelastic damped wave equation on compact Lie  groups}
\author{Arun Kumar Bhardwaj} 
\address{Arun Kumar Bhardwaj,  \endgraf Department of Mathematics
	\endgraf Indian Institute of Technology  Guwahati
	\endgraf Guwahati, Assam, India.} 
\email{arunkrbhardaj@gmail.com}
\author{Vishvesh Kumar} 

\address{Vishvesh Kumar, Ph. D.  \endgraf Department of Mathematics: Analysis, Logic and Discrete Mathematics
	\endgraf Ghent University
	\endgraf Krijgslaan 281, Building S8,	B 9000 Ghent,
	Belgium .} 
\email{vishveshmishra@gmail.com}

\author{Shyam Swarup Mondal} 

\address{Shyam Swarup Mondal  \endgraf Department of Mathematics
	\endgraf Indian Institute of Technology Delhi
	\endgraf Delhi, 110 016, India.} 
\email{mondalshyam055@gmail.com}

\date{\today}
\begin{document}
	\allowdisplaybreaks

	\begin{abstract} 
		Let $G$ be a compact Lie group.  In this article, we investigate the   Cauchy problem for a nonlinear  wave equation with the viscoelastic damping on   $G$. More preciously,  we investigate  some  $L^2$-estimates   for the solution to the homogeneous nonlinear viscoelastic damped wave equation on   $G$  utilizing    the group Fourier transform on $G$.   We also prove  that    there is no  improvement  of any decay rate for the norm $\|u(t,\cdot)\|_{L^2(G)}$  by further assuming the $L^1(G)$-regularity of initial data.  Finally,  using the noncommutative Fourier analysis on compact Lie groups,  we prove a local in time existence result in the energy space $\mathcal{C}^1([0,T],H^1_{\mathcal L}(G)).$ 
	\end{abstract}

	\maketitle
	%\tableofcontents 
	\section{Introduction}
	Let $G$ be a compact Lie group and let $\mathcal{L}$ be the Laplace-Beltrami operator on $G$ (which also coincides with the Casimir element of the enveloping algebra of the Lie algebra of $G$). In this paper  
	we derive decay estimates for the solution to   
	the Cauchy problem for a nonlinear wave equation with two types of damping terms, namely, 
	\begin{align} \label{eq0010}
		\begin{cases}
			\partial^2_tu-\mathcal{L}u+\partial_tu-\mathcal{L}\partial_tu=f(u), & x\in G,t>0,\\
			u(0,x)=\varepsilon u_0(x),  & x\in G,\\ \partial_tu(x,0)=\varepsilon u_1(x), & x\in G,
		\end{cases}
	\end{align}
	where  $\varepsilon$ is a positive constant describing the smallness of Cauchy data. Here,  for the moment, we assume that  $u_{0}$ and $ u_{1}$ are  taken from the energy space $ H_{\mathcal{L}}^1(G)$  and  concerning the nonlinearity  of $f(u)$, we shall deal only with  the typical case such as $f(u):=|u|^{p}, p>1$ without loosing the  essence of the problem. The equation \eqref{eq0010} is known as the viscoelastic dumped wave equation associted with the Laplace-Beltrami operators on compact Lie groups.

	The linear viscoelastic damped wave equation in the   setting of the Euclidean space has been well studied in the literature.    Several prominent  researchers have devoted considerable attention to the following Cauchy problem for   linear   damped wave equation
	\begin{align}\label{eq007}
		\begin{cases}
			\partial_{t}^{2} u-\Delta u+\partial_{t} u=0, & x \in \mathbb{R}^{n}, t>0, \\
			u(0, x)=u_{0}(x), \quad \partial_{t} u(0, x)=u_{1}(x), & x \in \mathbb{R}^{n},
		\end{cases}
	\end{align}	
	due to its  application of this model in the theory of viscoelasticity and some fluid dynamic. In his seminal work, 	Matsumura \cite{34} first  established basic decay estimates for the solution to the linear equation (\ref{eq007}) and afterthat,    many researchers    have concentrated on investigated  a typical important nonlinear problem, namely,   the following  semilinear damped wave equation
	\begin{align}\label{eq0088}
		\begin{cases}
			\partial_{t}^{2} u-\Delta u+\partial_{t} u=|u|^p, & x \in \mathbb{R}^{n}, t>0, \\
			u(0, x)=u_{0}(x), \quad \partial_{t} u(0, x)=u_{1}(x), & x \in \mathbb{R}^{n}.
		\end{cases}
	\end{align}	
	In this case,  there exists a real number $p_{F} \in(1, \infty)$ such that if $p>p_{F}$, then for some range of $p$ the corresponding Cauchy problem (\ref{eq0088}) has a small global in time solution $u(t, x)$ for the small initial data  $u_{0}$ and $ u_{1}$. On the other hand,     when $p \in\left(1, p_{F}\right]$,  under some    condition on the initial data ($\int_{\mathbb{R}^{n}} u_{i}(x) d x>0, i=0,1$), the corresponding problem (\ref{eq0088})   %$\int_{\mathbb{R}^{n}} u_{i}(x) d x>0, i=0,1,$
	does not have any nontrivial global solutions. In general,  such  a number $p_{F}$ is called as the Fujita critical exponent. For  a      detailed study related to Fujita   exponent, we  refer to \cite{2,5,4,7,8} and references therein.

	Further,  the study of the  semilinear damped wave equation (\ref{eq0088}) is further generalized  by the  following  strongly damped wave equation
	\begin{align}\label{eq00888}
		\begin{cases}
			\partial_{t}^{2} u-\Delta u+\Delta \partial_{t} u=\mu f(u) & x \in \mathbb{R}^{n}, t>0, \\
			u(0, x)=u_{0}(x), \quad \partial_{t} u(0, x)=u_{1}(x), & x \in \mathbb{R}^{n},
		\end{cases}
	\end{align}	
	by several researchers recently. When $\mu=0,$ in the case, for the dissipative structures of the Cauchy problem  (\ref{eq00888}),   Ponce \cite{15} and  Shibata  \cite{9}  derived some $L^{p}\left(\mathbb{R}^{n}\right)-L^{q}\left(\mathbb{R}^{n}\right)$ decay  estimates for the  solution to  (\ref{eq00888}) with $\mu=0.$  In the last  decade, some $L^{2}\left(\mathbb{R}^{n}\right)-L^{2}\left(\mathbb{R}^{n}\right)$ estimates with additional $L^{1}\left(\mathbb{R}^{n}\right)$-regularity were also derived  by several authors in \cite{10,11,13,14}.  In the same period, the authors of \cite{10} proved global (in time) existence of small data solution to the corresponding semilinear Cauchy problem to (\ref{eq00888}) with power nonlinearity on the right-hand side.  Recently, Ikehata-Todorova-Yordanov \cite{16} and Ikehata  \cite{17} have caught an asymptotic profile of solutions to problem  (\ref{eq00888})  which is well-studied in the field of the Navier-Stokes equation case.

	The study of the semilinear wave equation has also been extended in the non-Euclidean framework. Several papers are devoted for studying linear PDE in non-Euclidean structures in the last decades.   For example, the semilinear wave equation with or without damping has been investigated for the   Heisenberg group \cite{24,30}.    In the case of graded groups, we refer to the recent works \cite{gra1, gra2, gra3}.  	Concerning the damped wave equation on compact Lie groups, we refer to \cite{27, 28,31,garetto} (see also \cite{author} for the fractional wave equation). Here, we would also like to highlight that       estimates for the linear viscoelastic damped wave equation on the Heisenberg groupin was studied in  \cite{32}.

	Recently, Ikehata-Sawada \cite{19} and    Ikehata-Takeda \cite{20} considered and   
	studied the following Cauchy problem  which has two types of damping terms
	\begin{align}\label{eq008}
		\begin{cases}
			\partial_{t}^{2} u-\Delta u+\partial_{t} u-\Delta \partial_{t} u=0, & x \in \mathbb{R}^{n}, t>0, \\
			u(0, x)=u_{0}(x), \quad \partial_{t} u(0, x)=u_{1}(x), & x \in \mathbb{R}^{n}.
		\end{cases}
	\end{align}	
	Such type of related problem  with slight variants  extensively investigated in by authors \cite{14,21,22}.   
	
	An interesting and viable problem is to consider such types of (i.e., Cauchy problem \ref{eq1})   viscoelastic damped wave equations in the setting of non-Euclidean spaces, in particular, compact Lie groups.  So far to the best of our knowledge, in the framework of compact Lie  group, the viscoelastic damped wave equation  have not been studied yet. Our main aim of this  this article is to study the Cauchy problem   	
	for nonlinear wave equation with two types of damping terms on the compact Lie group  $G$, namely, 
	\begin{align*}  
		\begin{cases}
			\partial^2_tu-\mathcal{L}u+\partial_tu-\mathcal{L}\partial_tu=f(u), & x\in G,t>0,\\
			u(0,x)=\varepsilon u_0(x),  & x\in G,\\ \partial_tu(x,0)=\varepsilon u_1(x), & x\in G.
		\end{cases}
	\end{align*}

	\subsection{Main results}
	Throughout the paper  we denote $L^{q}(G)$, the space of $q$-integrable functions on $G$ with respect to the normalized Haar measure for $1 \leq  q<\infty$ (respectively, essentially bounded for $q=\infty$) and for $s>0$ and $q \in(1, \infty)$ the Sobolev space $H_{\mathcal{L} }^{ s, q}(G)$ is defined as the space
	\begin{align}\label{sob}
		H_{\mathcal{L}}^{s, q}(G) \doteq\left\{f \in L^{q}(G):(-\mathcal{L})^{s / 2} f \in L^{q}(G)\right\}
	\end{align}
	endowed with the norm $\|f\|_{H_{\mathcal{L}}^{s, q}(G)} \doteq\|f\|_{L^{q}(G)}+\left\|(-\mathcal{L})^{s / 2} f\right\|_{L^{q}(G)}$.  We simply denote $H_{\mathcal{L}}^{s}(G)$ as  the Hilbert space $H_{\mathcal{L}}^{s, 2}(G)$.   
	
	By employing the tools from the Fourier analysis for compact Lie groups, our first result below is concerned with the existence of the global solution to the homogeneous Cauchy problem (\ref{eq0010}) (i.e., when $f=0$) satisfying the suitable
	decay properties. More precisely, our goal is to derive  $L^2(G)$-decay estimates    for the Cauchy data as it is stated in the following    theorem.
	\begin{thm}\label{thm1}
		Let $u_0,~u_1\in H^1_{\mathcal L}(G)$ and let $u\in\mathcal{C}^1([0,\infty),H^1_{\mathcal L}(G))$ be the solution to the homogeneous Cauchy problem
		\begin{align}\label{eq1}
			\begin{cases}
				\partial^2_tu-\mathcal{L}u+\partial_tu-\mathcal{L}\partial_tu=0, & x\in G,~t>0,\\
				u(0,x)=u_0(x),  & x\in G\\ \partial_tu(x,0)=u_1(x), & x\in G.
			\end{cases}
		\end{align}
		Then, $u$ satisfies the following $L^2$-estimates
		\begin{align}\label{eq2}
			\| u(t,\cdot)\|_{L^2( G)} &\leq C\left(\| u_0\|_{L^2( G)}+\| u_1\|_{L^2( G)}\right),\\ \label{eq3}
			\|(-\mathcal L)^{1/2} u(t,\cdot)\|_{L^2( G)}&\leq C(1+t)^{-\frac{1}{2}}\left(\| u_0\|_{H^1_{\mathcal{L}}( G)}+\| u_1\|_{L^2( G)}\right),\\ \label{eq4}
			\|\partial_tu(t,\cdot)\|_{L^2( G)}&\leq C(1+t)^{-1}\left(\| u_0\|_{H^1_{\mathcal{L}}( G)}+\| u_1\|_{L^2( G)}\right),\\ \label{eq5}
			\|\partial_t(-\mathcal L)^{1/2} u(t,\cdot)\|_{L^2( G)}&\leq C(1+t)^{-\frac{3}{2}}\left(\| u_0\|_{H^1_{\mathcal{L}}( G)}+\| u_1\|_{H^1_{\mathcal{L}}( G)}\right),
		\end{align}
		for any $t\geq 0$, where $C$ is a positive multiplicative constant.
	\end{thm}
	\begin{rmk}
		From    the statement of Theorem \ref{thm1}  one can find that the regularity $u_1 \in  H^1_{\mathcal L}(G)$
		is necessary  to
		remove the singularity of $ \|\partial_t(-\mathcal L)^{1/2} u(t,\cdot)\|_{L^2( G)}$ near $t = 0$.
	\end{rmk}
	\begin{rmk}
		We also show that      there is no  improvement  of any decay rate for the norm $\|u(t,\cdot)\|_{L^2(G)}$  in Theorem  \ref{thm1} even if      we assume $L^1(G)$-regularity for $u_0$ and $u_1$. 
	\end{rmk}
	Next we prove the local well-posedness of the   Cauchy problem   (\ref{eq0010})  in the energy evolution space  $\mathcal{C}^1\left([0,T],  H^1_{ \mathcal L}(G)\right)$.   In particular, a Gagliardo-Nirenberg type inequality (proved in \cite{Gall}) will be used in order to estimate the power nonlinearity in $L^2(G)$.	 The following result  is about the   local existence   for the solution of  the Cauchy problem (\ref{eq0010}). 
	\begin{thm}\label{thm22}
		Let $G$ be a compact, connected Lie group and let $n$ be the topological dimension of $G.$ Assume that $n\geq3$. Suppose that $u_0,u_1\in  H^1_{\mathcal L}(G)$ and $p>1$ such that $p\leq\frac{n}{n-2}.$ Then, there exists $T=T(\varepsilon)>0$ such that the Cauchy problem (\ref{eq0010}) admits a uniquely determined mild solution $u$ in the space $ \mathcal{C}^1([0,T],H^1_{\mathcal L}(G)).$
	\end{thm}
	As in \cite{27},	we note that, in the statement of Theorem \ref{thm22}, the  restriction on the upper bound for the exponent $p$ which is $p\leq\frac{n}{n-2  }$   is necessary in order to apply Gagliardo-Nirenberg type inequality (\ref{eq34}) in  (\ref{f}) the the proof of Theorem \ref{thm22}.  The other restriction $n \geq 3$ is also technical and is  made to fulfill the assumptions for the employment of such inequality. This could be avoided if one look for solution in a different space such as $\mathcal{C}^1([0,T],H^s_{\mathcal L}(G)),\,s \in (0,1)$ than that of $\mathcal{C}^1([0,T],H^1_{\mathcal L}(G)).$

	It is customary to study the corresponding nonlinear homogeneous problem, that is, when $f=0$ prior to investigate  the  nonhomogneous problem (\ref{eq0010}). In this process,   we first establish a  $L^2$- energy estimates   for the solution to the homogeneous   viscoelastic damped wave equation on the compact Lie group $G$.  Having these estimates on our hand, we implement   a Gagliardo-Nirenberg type inequality  on compact Lie group (\cite{Gall, 27, 28, 31})  to prove  the local well-posedness result for the solution  to  (\ref{eq0010}). 
	We also show that,  even if     we assume $L^1(G)$-regularity for $u_0$ and $u_1$,    there is no  additional    decay rate can be gained for the $L^2$ norm of the solution of the corresponding homogeneous Cauchy problem.

	Apart from introduction the paper is organized as follows.  In Section \ref{sec2},  we recall some essentials from  the Fourier analysis on compact Lie groups which will be frequently used  throughout the paper. In Section \ref{sec4},  we prove Theorem \ref{thm1} by  deriving   some  $L^{2}$ decay estimates   for the solution of the homogeneous nonlinear viscoelastic damped wave equation on the compact Lie group $G$. We also show that,       there is no  additional gain in the  decay rate  of the $L^2$ norm of the solution to the corresponding homogeneous Cauchy problem even if  we   assume $L^1(G)$-regularity for $u_0$ and $u_1$ in     Section \ref{sec5}.   	Finally, in Section \ref{sec6},   we   briefly recall   the notion of mild solutions  in our framework and  prove   the local well-posedness of the   Cauchy problem   (\ref{eq0010})  in the energy evolution space  $\mathcal C^1\left([0,T],  H^\alpha_{\mathcal{L}}(G)\right)$.

	\subsection{Notations} 
	Throughout the article,  we use the following notations:  
	
	\begin{itemize}
		\item $f \lesssim g:$\,\,There exists a positive constant $C$ (whose value may change from line to line in this manuscript) such that $f \leq C g.$
		\item $G:$ Compact Lie group.
		\item $dx:$ The normalized Haar measure on the compact group $G.$
		\item $\mathcal{L}:$ The Laplace-Beltrami operator on $G.$
		\item $\mathbb{C}^{d \times d}:$ The set of matrices with complex entries of order $d.$
		\item $ \operatorname{Tr}(A)=\sum_{j=1}^{d} a_{j j}:$ The trace  of the matrix $A=\left(a_{i j}\right)_{1 \leq i, j \leq  d} \in \mathbb{C}^{d \times d}.$
		\item $I_{d} \in \mathbb{C}^{d \times d}:$  The identity matrix of order $d.$
	\end{itemize}

	\section{Preliminaries: Fourier analysis on compact Lie groups} \label{sec2}

	In this section, we recall some basics of the Fourier analysis on compact (Lie) groups to make the manuscript self-contained. A complete account of representation theory of the compact Lie  groups can be found in \cite{garetto, RT13, RuzT}. However, we mainly adopt the notation and terminology given in \cite{RuzT}.
	
	Let us first recall the definition of a representation of a compact group $G.$ A unitary representation  of $G$ is a pair $(\xi, \mathcal{H})$ such that the map $\xi:G \rightarrow U(\mathcal{H}),$ where $U(\mathcal{H})$ denotes the set of  unitary operators on complex Hilbert space $\mathcal{H},$  such that it satisfies following properties:
	\begin{itemize}
		\item The map $\xi$ is a group homomorphism, that is, $\xi(x y)=\xi(x)\xi(y).$
		\item The mapping $\xi:G \rightarrow U(\mathcal{H})$ is continuous with repsect to strong operator topology (SOT) on $U(\mathcal{H}),$  that is, the map $g \mapsto \xi(g)v$ is continuous for every $v \in \mathcal{H}.$ 
	\end{itemize}
	The Hilbert space $\mathcal{H}$ is called the representation space. If is there is no confusion, we just write $\xi$ for a representation $(\xi, \mathcal{H})$ of $G.$ 
	Two unitary representations $\xi, \eta$ of ${G}$ are called  equivalent if there exists an unitary operator, called intertwiner,  $T$ such that $T \xi(x)=\eta(x) T$ for any $x \in {G}$. 
	The intertwiner is a irreplaceable tool in the theory of representation of compact groups and helpful in the classification of representation.
	A (linear) subspace $V \subset \mathcal{H}$ is said to be invariant under the unitary representation $\xi$ of $G$ if $\xi(x) V \subset V$ for any $x \in {G}$. An irreducible unitary representation $\xi$ of $G$ is a representation such that the only closed and $\xi$-invariant subspaces of $\mathcal{H}$  are trivial once, that is, $\{0\}$ and the full space $ \mathcal{H}$. 
	
	The set of  all equivalence classes  $[\xi]$  of continuous irreducible unitary representations of $G$ is denoted by $\widehat{G}$ and called the unitary dual of $G.$ Since $G$ is compact, $\widehat{G}$ is a discrete set. It is known that an irreducible unitary representation $\xi$ of $G$ is finite dimensional, that is, the Hilbert space $\mathcal{H}$ is finite dimensional, say, $d_\xi$. Therefore, if we  choose a basis $\mathfrak{B}:=\{e_1,e_2,\ldots, e_{d_\xi}\}$ for the representation space $\mathcal{H}$ of $\xi$, we can identify $\mathcal{H}$ as $\mathbb{C}^{d_\xi}$ and consequently, we can view $\xi$ as a matrix-valued function $\xi: G \rightarrow U(\mathbb{C}^{d_{\xi} \times d_{\xi}})$, where $U(\mathbb{C}^{d_{\xi} \times d_{\xi}})$ denotes the space of all unitary matrices. The matrix coefficients $\xi_{ij}$ of the representation $\xi$ with respect to $\mathfrak{B}$ are given by $\xi_{ij}(x):=\langle \xi(x) e_j, e_i \rangle$ for all $i, j \in \{1,2, \ldots, d_\xi\}.$  It follows from the Peter-Weyl theorem that the set
	$$
	\left\{\sqrt{d_{\xi}} \xi_{i j}: 1 \leq i, j \leq d_{\xi},[\xi] \in \widehat{G}\right\}
	$$
	forms an orthonormal basis of $L^{2}(G)$.

	The group Fourier transform of $f \in L^1(G)$ at $\xi\in \widehat{G},$ denoted by $\widehat{f}(\xi),$ is defined by
	$$
	\widehat{f}(\xi):=\int_{G} f(x) \xi(x)^{*} d x,
	$$
	where $dx$ is the normalised Haar measure on $G$. It is apparent from the definition that  $\widehat{f}(\xi)$ is matrix valued and therefore, this definition can be interpreted as weak sense, that is, for $u,v \in \mathcal{H},$ we have
	$$ \langle	\widehat{f}(\xi)u, v \rangle:=\int_{G} f(x) \langle \xi(x)^{*}u, v \rangle d x.$$
	
	It follows from  the Peter-Weyl theorem that, for every $f \in L^2(G),$ we  have the following  the Fourier series representation:
	$$
	f(x)=\sum_{[\xi] \in \widehat{G}} d_{\xi} \operatorname{Tr}(\xi(x) \widehat{f}(\xi)).
	$$
	
	The Plancherel identity for the group Fourier transform on  $G$ takes the following  form
	\begin{align}\label{eq002}
		\|f\|_{L^{2}(G)}=\left(\sum_{[\xi] \in \widehat{G}} d_{\xi}\|\widehat{f}(\xi)\|_{\mathrm{HS}}^{2}\right)^{1 / 2}:=\|\widehat{f}\|_{\ell^2(\widehat{G})},
	\end{align}
	where $\|\cdot\|_{\mathrm{HS}}$ denotes the Hilbert-Schmidt norm of a matrix $A:=(a_{ij}) \in \mathbb{C}^{d_
		xi \times d_\xi}$ defined as
	$$	\|A\|_{\mathrm{HS}}^{2}=\operatorname{Tr}\left( A A^{*}\right)=\sum_{i, j=1}^{d_{\xi}}|a_{ij}|^2.$$
	
	We  would like to emphasize here that the  Plancherel identity  is one of the crucial tools to establish $L^2$-estimates of the solution to PDEs.

	Let $\mathcal{L}$ be the  Laplace-Beltrami operator on $G$. 	It is   important to understand the action of the group Fourier transform on the Laplace–Beltrami operator $\mathcal{L}$ for developing the machinery of the proofs.  For  $[\xi] \in \widehat{{G}}$,   the matrix elements   $\xi_{i j}$,  are  the eigenfunctions of $\mathcal{L}$ with the same   eigenvalue $-\lambda_{\xi}^{2}$. In other words, we have,    for any $ x \in {G},$
	$$
	-\mathcal{L} \xi_{i j}(x)=\lambda_{\xi}^{2} \xi_{i j}(x),   \qquad \text{for all } i, j \in\left\{1, \ldots, d_{\xi}\right\}.
	$$
	The symbol $\sigma_{\mathcal{L}}$ of  the  Laplace-Beltrami operator $\mathcal{L}$ on $G$ is given by 
	\begin{align}\label{symbol}
		\sigma_{\mathcal{L}}(\xi)=-\lambda_{\xi}^{2} I_{d_{\xi}},
	\end{align}
	for any $[\xi] \in \widehat{{G}}$ and   therefore, the following holds:  $$\widehat{\mathcal{L} f}(\xi)=\sigma_{\mathcal{L}}(\xi) \widehat{f}(\xi)=-\lambda_{\xi}^{2} \widehat{f}(\xi)$$ for any $[\xi] \in \widehat{ G}$. 
	
	For $s>0,$ the   Sobolev space $H_{\mathcal{L}}^s\left(G\right)$ of order $s$ is defined as follows: 
	$$H_{\mathcal{L}}^s(G):=\left\{u \in L^{2}(G):\|u\|_{H_{\mathcal{L}}^s(G)}<+\infty\right\},$$ where $\|u\|_{H_{\mathcal{L}}^s(G)}=\|u\|_{L^{2}(G)}+\left\|(-\mathcal{L})^{s / 2} u\right\|_{L^{2}({G})}$ and    $(-\mathcal{L})^{s / 2} $  is    defined in terms of the group Fourier transform by the follwoing formula
	$$(-\mathcal{L})^{\alpha / 2} f :=\mathcal{F}^{-1}\left(\lambda_{\xi}^{2 \alpha }(\mathcal{F} u)\right),  	\quad  \text{for all $[\xi] \in \widehat{{G}}$}.$$

	%	We will also use the following notation:
	%	$$
	%	\|u(t, \cdot)\|:=\|u(t, \cdot)\|_{H_{\mathcal{L}}^\alpha(G)}+\left\|\partial_{t} u(t, \cdot)\right\|_{L^{2}(G)} .
	%	$$
	
	%	For each $[\xi] \in \widehat{G}$, the matrix elements of $\xi$ are the eigenfunctions for the Laplacian $\mathcal{L}_{G}$ (or the Casimir element of the universal enveloping algebra), with the same eigenvalue which we denote by $-\lambda_{[\xi]}^{2}$, so that we have$$	-\mathcal{L}_{G} \xi_{i j}(x)=\lambda_{[\xi]}^{2} \xi_{i j}(x) \quad \text { for all } 1 \leq i, j \leq d_{\xi}	 	$$	The weight for measuring the decay or growth of Fourier coefficients in this setting is
	%	 	$$
	%	 	\langle\xi\rangle:=\left(1+\lambda_{[\xi]}^{2}\right)^{\frac{1}{2}},
	%	 	$$
	%	 	the eigenvalues of the elliptic first-order pseudo-differential operator ( $I$ $\left.\mathcal{L}_{G}\right)^{\frac{1}{2}}$.

	Further,  using Plancherel identity,  for any $s>0$, we have that
	$$
	\left\|(-\mathcal{L})^{s / 2} f\right\|_{L^{2}({G})}^{2}=\sum_{[\xi] \in \widehat{{G}}} d_{\xi} \lambda_{\xi}^{2 s}\|\widehat{f}(\xi)\|_{\mathrm{HS}}^{2}
	.	 	$$

	We also recall the definition of the space $\ell^\infty(\widehat G).$  We denote   $\mathcal{S}^\prime(\widehat G)$ as  the    space of slowly increasing distributions on the unitary dual $\widehat G$ of $G$. Then the space $\ell^\infty(\widehat G)$ is defined   as 
	$$\ell^\infty(\widehat G)=\{H=\{H([\xi])\}_{[\xi]\in\widehat G} : 	\|H\|_{\ell^\infty(\widehat G)} <\infty \},$$
	where 	$H([\xi])\in\mathbb C^{d_\xi\times d_\xi} ~\text{ for any } ~[\xi]\in\widehat G$ and 	\begin{align}\label{eql}
		\|H\|_{\ell^\infty(\widehat G)}:=\sup\limits_{[\xi]\in\widehat G}d_{\xi}^{-\frac{1}{2}}\|H(\xi)\|_{HS}<\infty. \end{align}
	Then   $\ell^\infty(\widehat G)$   is a subspace of $\mathcal{S}^\prime(\widehat G)$. 
	Moreover, for any $f\in L^1(G)$,  from the group Fourier transform it is true that 
	\begin{align}\label{eq30}
		\|\widehat f\|_{\ell^\infty(\widehat G)}\leq\| f\|_{L^1(G)}.
	\end{align}
	We must mention that implementation of (\ref{eq30}) very important in order to use the $L^1(G)$-regularity for the Cauchy data. A detailed study on the   construction  of     the space $\ell^\infty(\widehat G)$  can be found in  Section 10.3.2 of \cite{RuzT}  (see also Section 2.1.3 of \cite{viss}).

	\section{$L^2$-estimates for the solution to the homogeneous problem}  \label{sec4}
	In this section, we derive $L^2(G)– L^2(G)$ estimates for the solutions to    (\ref{eq1})  when $f=0$, namely, the homogeneous problem on $G:$
	\begin{align}\label{eq111}
		\begin{cases}
			\partial^2_tu-\mathcal{L}u+\partial_tu-\mathcal{L}\partial_tu=0, & x\in G, t>0,\\
			u(0,x)=u_0(x),  & x\in G,\\ \partial_tu(x,0)=u_1(x), & x\in G.
		\end{cases}
	\end{align}
	We  employ the group Fourier transform on the compact Lie group $G$ with respect to  the space variable $x$  together with the Plancherel identity in order to      estimate     $L^2$-norms of  $u(t, ·), (-\mathcal{L})^{\frac{1}{2}}u(t, \cdot)$,  $\partial_{t}u(t, ·)$, and $\partial_t (-\mathcal L)^{1/2} u(t,\cdot)$. 
	
	Let $u$ be a solution to (\ref{eq111}). Let $\widehat{u}(t, \xi)=(\widehat{u}(t, \xi)_{kl})_{1\leq k, l\leq d_\xi}\in \mathbb{C}^{d_\xi\times d_\xi}, [\xi]\in\widehat{ G}$ denote the Fourier transform of $u$  with respect to the $x $ variable. Invoking the group Fourier transform with respect to $x$ on   (\ref{eq111}), we deduce that $\widehat{u}(t, \xi)$ is  a solution to the following  Cauchy problem for the system of ODE's (with size of the system that depends on the representation $\xi$)

	\begin{align}\label{eq66}
		\begin{cases}
			\partial^2_t\widehat{u}(t,\xi)-	\sigma_{\mathcal{L}}(\xi)\widehat{u}(t,\xi)+\partial_t\widehat{u}(t,\xi)-\sigma_{\mathcal{L}}(\xi) \partial_t\widehat{u}=0,& [\xi]\in\widehat{ G},t>0,\\ \widehat{u}(0,\xi)=\widehat{u}_0(\xi), &[\xi]\in\widehat{ G},\\ \partial_t\widehat{u}(0,\xi)=\widehat{u}_1(\xi), &[\xi]\in\widehat{ G},
		\end{cases} 
	\end{align}
	where  $\sigma_{\mathcal{L}}$	is the symbol of the  of the Laplace-Beltrami operator operator $\mathcal{L}$ defined in (\ref{symbol}). 
	Using the identity (\ref{symbol}),  the   system  (\ref{eq66}) is decoupled in $d_\xi^2$ independent   ODEs, namely,
	
	%	\begin{align}\label{eq6}
		%	\begin{cases}
			%		\partial^2_t\widehat{u}(t,\xi)+(1+\lambda_\xi^2)\partial_t\widehat{u}(t,\xi)+\lambda^2_\xi\widehat{u}=0,& [\xi]\in\widehat{ G},~t>0,\\ \widehat{u}(0,\xi)=\widehat{u}_0(\xi), &[\xi]\in\widehat{ G}\\ \partial_t\widehat{u}(0,\xi)=\widehat{u}_1(\xi), &[\xi]\in\widehat{ G}.
			%	\end{cases} 
		%\end{align}
		%		Now, the system (\ref{eq6}) is decoupled into $d_\xi^2$ ODEs, namely,
		\begin{align}\label{eq7}
			\begin{cases}
				\partial^2_t\widehat{u}(t,\xi)_{kl}+(1+\lambda_\xi^2)\partial_t\widehat{u}(t,\xi)_{kl}+\lambda^2_\xi\widehat{u}(t,\xi)_{kl}=0,& [\xi]\in\widehat{ G},t>0,\\ \widehat{u}(0,\xi)_{kl}=\widehat{u}_0(\xi)_{kl}, &[\xi]\in\widehat{ G},\\ \partial_t\widehat{u}(0,\xi)_{kl}=\widehat{u}_1(\xi)_{kl}, &[\xi]\in\widehat{ G},
			\end{cases}
		\end{align}
		for all $k,l\in\{1,2,\ldots,d_\xi\}.$\\
		Then,    the characteristic equation of (\ref{eq7}) is given by
		\[\lambda^2+(1+\lambda_\xi^2)\lambda+\lambda_\xi^2=0,\]
		and consequently the characteristic roots of (\ref{eq7}) are  
		\[\lambda=\frac{-(1+\lambda_\xi^2)\pm|  1-\lambda_\xi^2| }{2}.\]
		We note that if $\lambda_\xi^2\neq1,$ then there are two distinct roots, say, $\lambda^+=-1$ and $\lambda^-=-\lambda_\xi^2,$ and if $\lambda_\xi^2=1$ then both the roots are same and equal to $\lambda=-1.$  We   analyze the following two cases for   the solution to the system (\ref{eq7}).
		\vspace{.5cm}
		
		\noindent\textbf{Case I.} Let $\lambda_\xi^2\neq1.$ The solution of (\ref{eq7}) is given by
		\begin{align}\label{eq8}
			\widehat{u}(t,\xi)_{kl}=\mathcal{K}_0(t,\xi)\widehat{u}_0(\xi)_{kl}+\mathcal{K}_1(t,\xi)\widehat{u}_1(\xi)_{kl},
		\end{align}
		where
		\begin{align}\label{eq9}
			\begin{cases}\mathcal{K}_0(t,\xi)=\frac{e^{-\lambda_\xi^2t}-\lambda_\xi^2e^{-t}}{1-\lambda_\xi^2},\\  \mathcal{K}_1(t,\xi)=\frac{e^{-\lambda_\xi^2t}-e^{-t}}{1-\lambda_\xi^2}.
			\end{cases}
		\end{align}
		
		\noindent\textbf{Case II.} Let $\lambda_\xi^2=1.$ The solution of (\ref{eq7}) is given by
		\begin{align}\label{eq10}
			\widehat{u}(t,\xi)_{kl}=\mathcal{K}_0(t,\xi)\widehat{u}_0(\xi)_{kl}+\mathcal{K}_1(t,\xi)\widehat{u}_1(\xi)_{kl},
		\end{align}
		where
		\begin{align}\label{eq11}
			\begin{cases}
				\mathcal{K}_0(t,\xi)=(1+t)e^{-t},\\  \mathcal{K}_1(t,\xi)=te^{-t}.
			\end{cases}
		\end{align}
		Thus we have
		\begin{align}\label{eq255}
			\widehat{u}(t,\xi)_{kl}=
			\begin{cases}
				\frac{e^{-\lambda_\xi^2t}-\lambda_\xi^2e^{-t}}{1-\lambda_\xi^2}\widehat{u}_0(\xi)_{kl}+\frac{e^{-\lambda_\xi^2t}-e^{-t}}{1-\lambda_\xi^2} \widehat{u}_1(\xi)_{kl}, & \lambda_\xi^2\neq1,\\ (1+t)e^{-t} \widehat{u}_0(\xi)_{kl}+te^{-t}\widehat{u}_1(\xi)_{kl}, & \lambda_\xi^2=1.
			\end{cases}
		\end{align}
		Also we note that 
		$$
		\partial_t^\ell\mathcal{K}_0(t,\xi)=\frac{(-\lambda_\xi^2)^\ell e^{-\lambda_\xi^2t}+(-1)^{\ell+1}\lambda_\xi^2e^{-t}}{1-\lambda_\xi^2}, $$ and$$ \partial_t^\ell\mathcal{K}_1(t,\xi)=\frac{(-\lambda_\xi^2)^\ell e^{-\lambda_\xi^2t}+(-1)^{\ell+1}e^{-t}}{1-\lambda_\xi^2}.
		$$
		%		Now we define the evolution operator $K_0(t)g$ and $K_1(t)g$ as \[K_j(t)g:=\mathcal{F}^{-1}[\mathcal{K}_j(t)\widehat{g}],\] for $j=0,~1. $
		%		Now we are in a position to     prove Theorem \ref{thm1}.
		
		%	\section{Proof of theorem \ref{thm1}}
		%	 
		First we determine an explicit expression for the $L^2(G)$ norms of $u(t,\cdot), (-\mathcal L)^{1/2} u(t, \cdot),  \partial_t u(t,\cdot),$ 
		and $\partial_t (-\mathcal L)^{1/2} u(t,\cdot)$. We apply  the group Fourier transform with respect to the spatial variable $x$ together with the Plancherel identity  in order to determine  the  $L^2(G)$ norms.
		
		To simplify the presentation we introduce the following partition of the unitary dual $\widehat{ G}$ as:
		\begin{align*}
			\mathcal{R}_1&=\{[\xi]\in\widehat{ G}:\lambda_\xi^2=0\},\\ \mathcal{R}_2&=\{[\xi]\in\widehat{ G}:0<\lambda_\xi^2<1\},\\ \mathcal{R}_3&=\{[\xi]\in\widehat{ G}:\lambda_\xi^2=1\},\text{ and }\\
			\mathcal{R}_4&=\{[\xi]\in\widehat{ G}:\lambda_\xi^2>1\}. 
		\end{align*}
		Here we note that some of the above sets  may be empty.   
		\subsection{Estimate for $\| u(t,\cdot)\|_{L^2( G)}$}
		By Plancherel formula we have 
		\begin{align}\label{eq12}
			\| u(t,\cdot)\|_{L^2( G)}&=\sum\limits_{[\xi]\in\widehat{ G}}d_\xi\sum\limits_{k,l=1}^{d_{\xi}}|  \widehat{u}(t,\xi)_{kl}| ^2.
		\end{align}
		\noindent\textbf{Estimate on $\mathcal{R}_1.$} Using $\lambda_\xi^2=0$ in (\ref{eq9}) we get
		\begin{align}\label{eq13}
			| \mathcal{K}_0(t,\xi)| ,~| \mathcal{K}_1(t,\xi)| \lesssim 1.
		\end{align}
		Hence (\ref{eq8}) implies that
		\begin{align}\label{eq14}
			| \widehat{u}(t,\xi)_{kl}| \lesssim
			| \widehat{u}_0(\xi)_{kl}| +| \widehat{u}_1(\xi)_{kl}| .
		\end{align}
		\noindent\textbf{Estimate on $\mathcal{R}_2.$} Since the set $\{\lambda_\xi^2\}_{[\xi]\in\widehat{ G}}$ is a discrete set, there exist $\delta_1$ and $\delta_2$ such that
		\begin{align}\label{eqA}
			0<\delta_1\leq\lambda_\xi^2\leq\delta_2<1,\qquad[\xi]\in\mathcal{R}_2,
		\end{align} 
		consequently $\frac{1}{1-\lambda_\xi^2}$ is bounded on $\mathcal{R}_2$ and by (\ref{eq9}) we have
		\begin{align}\label{eq15}
			| \mathcal{K}_0(t,\xi)| ,~| \mathcal{K}_1(t,\xi)| \lesssim e^{-\delta_1t}.
		\end{align}
		Hence by (\ref{eq8}) we get
		\begin{align}\label{eq16}
			| \widehat{u}(t,\xi)_{kl}| \lesssim e^{-\delta_1t}\left[| \widehat{u}_0(\xi)_{kl}| +| \widehat{u}_1(\xi)_{kl}| \right].
		\end{align}
		\noindent\textbf{Estimate on $\mathcal{R}_3.$} By (\ref{eq11}) we have
		\begin{align}\label{eq17}
			| \mathcal{K}_0(t,\xi)| ,~| \mathcal{K}_1(t,\xi)| \lesssim (1+t)e^{-t}.
		\end{align}		
		Hence by (\ref{eq10}) we get
		\begin{align}\label{eq18}
			| \widehat{u}(t,\xi)_{kl}| \lesssim (1+t)e^{-t}\left[| \widehat{u}_0(\xi)_{kl}| +| \widehat{u}_1(\xi)_{kl}| \right].
		\end{align}
		\noindent\textbf{Estimate on $\mathcal{R}_4.$} Again discreteness of the set $\{\lambda_\xi^2\}_{[\xi]\in\widehat{ G}}$ implies that there exists $\delta_3$ such that \begin{align}\label{eqB}
			1<\delta_3\leq\lambda_\xi^2, \qquad[\xi]\in\mathcal{R}_4.
		\end{align}Hence $\frac{1}{\lambda_\xi^2-1}$ and $\frac{\lambda_\xi^2}{\lambda_\xi^2-1}$ are bounded on $\mathcal{R}_4,$ consequently (\ref{eq9}) yields \begin{align}\label{eq19}
			| \mathcal{K}_0(t,\xi)| ,~| \mathcal{K}_1(t,\xi)| \lesssim e^{-t}.
		\end{align}
		Using (\ref{eq8}) we obtain
		\begin{align}\label{eq20}
			| \widehat{u}(t,\xi)_{kl}| \lesssim e^{-t}\left[| \widehat{u}_0(\xi)_{kl}| +| \widehat{u}_1(\xi)_{kl}| \right].
		\end{align}
		Combining (\ref{eq14}), (\ref{eq16}), (\ref{eq18}), and (\ref{eq20}) we get
		\begin{align}\label{eq21}
			| \widehat{u}(t,\xi)_{kl}| \lesssim
			| \widehat{u}_0(\xi)_{kl}| +| \widehat{u}_1(\xi)_{kl}| ,\qquad[\xi]\in\widehat{ G}.
		\end{align}
		Substituting (\ref{eq21}) in (\ref{eq12}) we obtain 
		\begin{align}\label{eq2111}	
			\| u(t,\cdot)\|_{L^2( G)} &\leq C\left(\| u_0\|_{L^2( G)}+\| u_1\|_{L^2( G)}\right) .\end{align}
		\smallskip
		
		\noindent\textit{Remark.} Note that we do not get any decay on the R.H.S of (\ref{eq2111}) due to the fact that the set $\mathcal{R}_1$ is always non empty (in fact singleton) .

		\subsection{Estimate for $\|(-\mathcal{L})^{1/2} u(t,\cdot)\|_{L^2(G)}$}
		By Plancherel formula we get
		\[\|(-\mathcal{L})^{1/2} u(t,\cdot)||_{L^2(G)}=\sum\limits_{[\xi]\in\widehat{G}}d_\xi\|\sigma_{(-\mathcal{L})^{1/2}}(\xi)\widehat{u}(t,\xi)\|_{\operatorname{HS}}^2=\sum\limits_{[\xi]\in\widehat{G}}d_\xi\sum\limits_{k,l=1}^{d_\xi}\lambda_\xi^2|\widehat{u}(t,\xi)_{kl}|^2.\]
		\smallskip
		
		\noindent\textbf{Estimate on $\mathcal{R}_1.$} We have
		\[\lambda_\xi^2|\widehat{u}(t,\xi)_{kl}|^2=0.\]
		\smallskip
		
		\noindent\textbf{Estimate on $\mathcal{R}_2.$} By (\ref{eq15}) we obtain
		\begin{align}\label{eq22}
			\lambda_\xi^2|\widehat{u}(t,\xi)_{kl}|^2\lesssim e^{-2\delta_1t}[|\widehat{u}_0(\xi)_{kl}|^2+|\widehat{u}_1(\xi)_{kl}|^2].
		\end{align}
		\smallskip
		
		\noindent\textbf{Estimate on $\mathcal{R}_3.$} by (\ref{eq17}) we have
		\begin{align}\label{eq23}
			\lambda_\xi^2|\widehat{u}(t,\xi)_{kl}|^2\lesssim (1+t)^2e^{-2t}\left[|\widehat{u}_0(\xi)_{kl}|^2+|\widehat{u}_1(\xi)_{kl}|^2\right].
		\end{align}
		\smallskip
		
		\noindent\textbf{Estimate on $\mathcal{R}_4.$} Again using the fact that $\frac{1}{\lambda_\xi^2-1}$ and $\frac{\lambda_\xi^2}{\lambda_\xi^2-1}$ are bounded on $\mathcal{R}_4$ we obtain
		\begin{align}\label{eq24}
			\lambda_\xi^2|\widehat{u}(t,\xi)_{kl}|^2\lesssim e^{-2t}\left[\lambda_\xi^2|\widehat{u}_0(\xi)_{kl}|^2+|\widehat{u}_1(\xi)_{kl}|^2\right].
		\end{align}
		Therefore, 
		\begin{align}\label{eqnew}\nonumber
			\|(-\mathcal{L})^{1/2} u(t,\cdot)\|^2_{L^2(G)}&=\sum\limits_{[\xi]\in\mathcal{R}_1}d_\xi\sum\limits_{k,l=1}^{d_\xi}\lambda_\xi^2|\widehat{u}(t,\xi)_{kl}|^2+\sum\limits_{[\xi]\in\mathcal{R}_2}d_\xi\sum\limits_{k,l=1}^{d_\xi}\lambda_\xi^2|\widehat{u}(t,\xi)_{kl}|^2\\\nonumber&\qquad +\sum\limits_{[\xi]\in\mathcal{R}_3}d_\xi\sum\limits_{k,l=1}^{d_\xi}\lambda_\xi^2|\widehat{u}(t,\xi)_{kl}|^2+\sum\limits_{[\xi]\in\mathcal{R}_4}d_\xi\sum\limits_{k,l=1}^{d_\xi}\lambda_\xi^2|\widehat{u}(t,\xi)_{kl}|^2\\\nonumber
			&\lesssim e^{-2\delta_1t}\sum\limits_{[\xi]\in\mathcal{R}_2}d_\xi\sum\limits_{k,l=1}^{d_\xi}\left(|\widehat{u}_0(\xi)_{kl}|^2+|\widehat{u}_1(\xi)_{kl}|^2\right)\\\nonumber&\qquad +(1+t)^2e^{-2t}\sum\limits_{[\xi]\in\mathcal{R}_3}d_\xi\sum\limits_{k,l=1}^{d_\xi}\left(|\widehat{u}_0(\xi)_{kl}|^2+|\widehat{u}_1(\xi)_{kl}|^2\right)\\\nonumber&\qquad+e^{-2t}\sum\limits_{[\xi]\in\mathcal{R}_4}d_\xi\sum\limits_{k,l=1}^{d_\xi}\left(\lambda_\xi^2|\widehat{u}_0(\xi)_{kl}|^2+|\widehat{u}_1(\xi)_{kl}|^2\right)\\\nonumber&\lesssim (1+t)^2e^{-2\delta_1t}\left(\|u_0\|_{H^1_{\mathcal{L}}(G)}^2+\|u_1\|^2_{L^2(G)}\right)\\&\lesssim (1+t)^{-1}\left(\|u_0\|_{H^1_{\mathcal{L}}(G)}^2+\|u_1\|^2_{L^2(G)}\right). 
		\end{align}
		\subsection{Estimate for $\|\partial_tu(t,\cdot)\|_{L^2(G)}$} By Plancherel theorem we have
		\begin{align*}
			\|\partial_tu(t,\cdot)\|^2_{L^2(G)}=\sum\limits_{[\xi]\in\widehat{G}}d_\xi\sum\limits_{k,l=1}^{d_\xi}|\partial_t\widehat{u}(t,\xi)_{kl}|^2.
		\end{align*}
		We note that 
		\begin{align}\label{eq25}
			\partial_t\widehat{u}(t,\xi)_{kl}=
			\begin{cases}
				\frac{\lambda_\xi^2}{1-\lambda_\xi^2}(e^{-t}-e^{-\lambda_\xi^2t})\widehat{u}_0(\xi)_{kl}+\frac{e^{-t}-\lambda_\xi^2e^{-\lambda_\xi^2t}}{1-\lambda_\xi^2}\widehat{u}_1(\xi)_{kl}, & \lambda_\xi^2\neq1,\\ -te^{-t}\widehat{u}_0(\xi)_{kl}+(1-t)e^{-t}\widehat{u}_1(\xi)_{kl}, & \lambda_\xi^2=1.
			\end{cases}
		\end{align}
		\smallskip
		
		\noindent\textbf{Estimate on $\mathcal{R}_1.$} From  (\ref{eq25}),
		we have 	\[|\partial_t\widehat{u}(t,\xi)_{kl}|=e^{-t}|\widehat{u}_1(\xi)_{kl}|.\]
		\smallskip
		
		\noindent\textbf{Estimate on $\mathcal{R}_2.$} By (\ref{eq25}) and (\ref{eqA}) we obtain
		\[|\partial_t\widehat{u}(t,\xi)_{kl}|\lesssim e^{-\delta_1t}\left(|\widehat{u}_0(\xi)_{kl}|+|\widehat{u}_1(\xi)_{kl}|\right).\]
		\smallskip
		
		\noindent\textbf{Estimate on $\mathcal{R}_3.$} By (\ref{eq25}) we get
		\[|\partial_t\widehat{u}(t,\xi)_{kl}|\lesssim e^{-t}\left(t|\widehat{u}_0(\xi)_{kl}|+|1-t||\widehat{u}_1(\xi)_{kl}|\right)\lesssim (1+t)e^{-t}\left(|\widehat{u}_0(\xi)_{kl}|+|\widehat{u}_1(\xi)_{kl}|\right).\]
		\smallskip
		
		\noindent\textbf{Estimate on $\mathcal{R}_4.$} Using (\ref{eq25}), (\ref{eqB}) and the fact that $\frac{1}{\lambda_\xi^2-1}$ and $\frac{\lambda_\xi^2}{\lambda_\xi^2-1}$ are bounded on $\mathcal{R}_4$ we obtain
		\[|\partial_t\widehat{u}(t,\xi)_{kl}|\lesssim e^{-t}\left(|\widehat{u}_0(\xi)_{kl}|+|\widehat{u}_1(\xi)_{kl}|\right).\]
		Combining these we get 
		
		\[|\partial_t\widehat{u}(t,\xi)_{kl}|\lesssim (1+t)e^{-\delta_1t}\left(|\widehat{u}_0(\xi)_{kl}|+|\widehat{u}_1(\xi)_{kl}|\right),~\forall~[\xi]\in\widehat{G}.\]
		Thus 
		\begin{align}\label{eqneww}\nonumber
			\|\partial_tu(t,\cdot)\|^2_{L^2(G)}&\lesssim (1+t)^2e^{-2\delta_1t}\left(\|u_0\|^2_{L^2(G)}+\|u_1\|^2_{L^2(G)}\right)\\
			&\lesssim (1+t)^{-2}\left(\|u_0\|^2_{L^2(G)}+\|u_1\|^2_{L^2(G)}\right) .
		\end{align} 
		
		\subsection{Estimate for $\|\partial_t(-\mathcal{L})^{1/2}u(t,\cdot)\|_{L^2(G)}$} By Plancherel theorem we have
		\begin{align*}
			\|\partial_t(-\mathcal{L})^{1/2}u(t,\cdot)\|^2_{L^2(G)}=\sum\limits_{[\xi]\in\widehat{G}}d_\xi\sum\limits_{k,l=1}^{d_\xi}\lambda_\xi^2|\partial_t\widehat{u}(t,\xi)_{kl}|^2.
		\end{align*}
		\noindent\textbf{Estimate on $\mathcal{R}_1.$} We have
		\[\lambda_\xi^2|\partial_t\widehat{u}(t,\xi)_{kl}|^2=0.\]
		\smallskip
		
		\noindent\textbf{Estimate on $\mathcal{R}_2.$} By (\ref{eq15}) and (\ref{eq25}) we obtain
		\begin{align}\label{eq28}
			\lambda_\xi^2|\partial_t\widehat{u}(t,\xi)_{kl}|^2\lesssim e^{-2\delta_1t}[|\widehat{u}_0(\xi)_{kl}|^2+|\widehat{u}_1(\xi)_{kl}|^2].
		\end{align}
		\smallskip
		
		\noindent\textbf{Estimate on $\mathcal{R}_3.$} By (\ref{eq25}) we have
		\begin{align}\label{eq29}
			\lambda_\xi^2|\partial_t\widehat{u}(t,\xi)_{kl}|^2\lesssim (1+t)^2e^{-2t}\left[|\widehat{u}_0(\xi)_{kl}|^2+|\widehat{u}_1(\xi)_{kl}|^2\right].
		\end{align}
		\smallskip
		
		\noindent\textbf{Estimate on $\mathcal{R}_4.$} Using (\ref{eq25}), (\ref{eqB}) and the fact that $\frac{1}{\lambda_\xi^2-1}$ and $\frac{\lambda_\xi^2}{\lambda_\xi^2-1}$ are bounded on $\mathcal{R}_4$ we obtain
		\[\lambda_\xi^2|\partial_t\widehat{u}(t,\xi)_{kl}|\lesssim e^{-t}\lambda_\xi^2\left(|\widehat{u}_0(\xi)_{kl}|+|\widehat{u}_1(\xi)_{kl}|\right).\]	
		Therefore,
		
		\begin{align}\label{eqnewww}\nonumber
			\|\partial_t(-\mathcal{L})^{1/2} u(t,\cdot)\|^2_{L^2(G)}&=\sum\limits_{[\xi]\in\mathcal{R}_1}d_\xi\sum\limits_{k,l=1}^{d_\xi}\lambda_\xi^2|\widehat{u}(t,\xi)_{kl}|^2+\sum\limits_{[\xi]\in\mathcal{R}_2}d_\xi\sum\limits_{k,l=1}^{d_\xi}\lambda_\xi^2|\widehat{u}(t,\xi)_{kl}|^2\\\nonumber&+\sum\limits_{[\xi]\in\mathcal{R}_3}d_\xi\sum\limits_{k,l=1}^{d_\xi}\lambda_\xi^2|\widehat{u}(t,\xi)_{kl}|^2+\sum\limits_{[\xi]\in\mathcal{R}_4}d_\xi\sum\limits_{k,l=1}^{d_\xi}\lambda_\xi^2|\widehat{u}(t,\xi)_{kl}|^2\\\nonumber
			&\lesssim e^{-2\delta_1t}\sum\limits_{[\xi]\in\mathcal{R}_2}d_\xi\sum\limits_{k,l=1}^{d_\xi}\left(|\widehat{u}_0(\xi)_{kl}|^2+|\widehat{u}_1(\xi)_{kl}|^2\right)\\\nonumber&+(1+t)^2e^{-2t}\sum\limits_{[\xi]\in\mathcal{R}_3}d_\xi\sum\limits_{k,l=1}^{d_\xi}\left(|\widehat{u}_0(\xi)_{kl}|^2+|\widehat{u}_1(\xi)_{kl}|^2\right)\\\nonumber&+e^{-2t}\sum\limits_{[\xi]\in\mathcal{R}_4}d_\xi\sum\limits_{k,l=1}^{d_\xi}\lambda_\xi^2\left(|\widehat{u}_0(\xi)_{kl}|^2+|\widehat{u}_1(\xi)_{kl}|^2\right)\\\nonumber&\lesssim (1+t)^2e^{-2\delta_1t}\left(\|u_0\|_{H^1_{\mathcal{L}}(G)}^2+\|u_1\|^2_{H^1_{\mathcal{L}}(G)}\right)\\&\lesssim (1+t)^{-3}\left(\|u_0\|_{H^1_{\mathcal{L}}(G)}^2+\|u_1\|^2_{H^1_{\mathcal{L}}(G)}\right). 
		\end{align}
		Now we are in a position to   prove  Theorem \ref{thm1}. 
		\begin{proof}[Proof of Theorem \ref{thm1}]
			The proof  follows from  the       estimates  (\ref{eq2111}), (\ref{eqnew}), (\ref{eqneww}), and  (\ref{eqnewww})  for $\|u(t, \cdot )\|_{L^{2}(G)}$, $\left\|(-\mathcal{L})^{1 / 2} u(t, \cdot )\right\|_{L^{2}(G)}$, 	$\left\|\partial_{t} u(t, \cdot )\right\|_{L^{2}(G)}$, and $	\|\partial_t(-\mathcal{L})^{1/2} u(t,\cdot)\|_{L^2(G)}$,  respectively.
		\end{proof}
		
		\section{$L^1( G)-L^2( G)$ estimates for the solution to the homogeneous problem}\label{sec5}
		In this section we  show that        there is no  improvement  of any decay rate for the norm $\|u(t,\cdot)\|_{L^2(G)}$  when further  we assume $L^1(G)$-regularity for $u_0$ and $u_1$. Note that in Theorem \ref{thm1} we employed data on $L^2( G)$ basis. Since $G$ is compact group it follows that the Haar measure of $G$ is finite. This implies that $L^2(G)$ is continuously embedded in $L^1(G)$ and therefore, one might be curious to know that  which changes will occur   if we further  implement $L^1(G)$-regularity for $u_0$ and $u_1.$

		From (\ref{eq16}), (\ref{eq18}), and (\ref{eq20}) it immediately follows that
		\begin{align*}
			\sum\limits_{[\xi]\in\widehat G\backslash \mathcal{R}_1}d_\xi\sum\limits_{k,l=1}^{d_{\xi}}|\widehat{u}(t,\xi)_{kl}|^2&\lesssim (1+t)^2e^{-2\delta_1t}\sum\limits_{[\xi]\in\widehat G\backslash\mathcal{R}_1}d_\xi\sum\limits_{k,l=1}^{d_{\xi}}\left(|\widehat{u}_0(\xi)_{kl}|^2+|\widehat{u}_1(\xi)_{kl}|^2\right)\\&\lesssim (1+t)^2e^{-2\delta_1t}\left(\|\widehat u_0(\xi)_{kl}\|_{L^2(G)}^2+\|\widehat u_1(\xi)_{kl}\|_{L^2(G)}^2\right)
		\end{align*}
		for some suitable constant $\delta_1.$ Therefore, the contribution to the sum in (\ref{eq12}) corresponding to  $\mathcal R_1$ refrain us to get a decay rate for $\|u(t,\cdot)\|^2_{L^2(G)}.$ Thus, if we want to employ $L^1(G)$-regularity rather than $L^2(G)$-regularity, then we must apply it to obtain the estimation of the terms with $[\xi]\in\mathcal R_1.$ Here, we must note that the set $\mathcal R_1$ is a singleton.
		
		Note that for the multiplier in (\ref{eq9}), 	  the best estimate that one can obtain on the set $\mathcal R_1$  are 
		\begin{align*}
			\mid\mathcal{K}_0(t,\xi)\mid,~\mid\mathcal{K}_1(t,\xi)\mid\lesssim 1.
		\end{align*}
		Since the set $\mathcal R_1$ is singleton, using the definition defined in (\ref{eql}),  we obtain
		\begin{align*}
			\sum\limits_{[\xi]\in\mathcal{R}_1}d_\xi\sum\limits_{k,l=1}^{d_{\xi}}|\widehat{u}(t,\xi)_{kl}|^2&\lesssim d_\xi\sum\limits_{k,l=1}^{d_{\xi}}\left(|\widehat{u}_0(\xi)_{kl}|^2+|\widehat{u}_1(\xi)_{kl}|^2\right)\\&\lesssim d_\xi\left(\|\widehat{u}_0(\xi)\|^2_{\operatorname{HS}}+\|\widehat{u}_1(\xi)\|^2_{\operatorname{HS}}\right)\\&\lesssim  \left(\sup\limits_{[\xi]\in\widehat G}d_\xi^{-\frac{1}{2}}\left(\|\widehat{u}_0(\xi)\|_{\operatorname{HS}}+\|\widehat{u}_1(\xi)\|_{\operatorname{HS}}\right)\right)^2\\&\lesssim\left(\|\widehat{u}_0(\xi)_{kl}\|_{\ell^\infty(\widehat G)}+\|\widehat{u}_1(\xi)\|_{\ell^\infty(\widehat G)}\right)^2\\&\lesssim \left(\|\widehat u_0(\xi)_{kl}\|_{L^1(G)}^2+\|\widehat u_1(\xi)_{kl}\|_{L^1(G)}^2\right).
		\end{align*}
		This shows that even if we use $L^1(G)$-regularity we are not able to get any decay rate for the norm $\|u(t,\cdot)\|_{L^2(G)}.$ 
		
		One can easily observe that the main reason behind this behaviour is that we can not neglect the eigenvalue $0$ as the Plancherel measure on a compact Lie group turns out to be a weighted counting measure.

		\begin{rmk}
			In the noncompact setting such as the Euclidean space and the Heisenberg group, one can get a global existence result for a non empty range for $p$ by asking an additional $L^{1}$-regularity for the initial data. Consequently, we get an improved  decay rates for the estimates of the $L^{2}$-norm of the solution to the corresponding linear homogeneous problem. One can see \cite{27,32} for the illustration and discussion on this matter.
			
		\end{rmk} 
		\section{Local existence}\label{sec6}
		This section is devoted to  prove  Theorem \ref{thm22}, i.e.,   the local well-posedness of the   Cauchy problem   (\ref{eq0010})  in the energy evolution space  $ \mathcal{C}^1\left([0, T], H_{\mathcal{L}}^{1}(G)\right)$. To present the proof of Theorem \ref{thm22}, first we recall    the notion of mild solutions  in our setting. 
		
		Consider the space \[X(T):=\mathcal{C}^1 \left([0,T],  H^1_{\mathcal L}(G)\right) ,\] equipped with the norm
		\begin{align}\label{eq33333}\nonumber
			\|u\|_{X(T)}&:=\sup\limits_{t\in[0,T]}(\|u(t,\cdot)\|_{L^2(G)}+\|(-\mathcal L)^{1/2}u(t,\cdot)\|_{L^2(G)}+\|\partial_tu(t,\cdot)\|_{L^2(G)}\\&\qquad +\|\partial_t(-\mathcal L)^{1/2}u(t,\cdot)\|_{L^2(G)}).
		\end{align}
		The solution to the nonlinear inhomogeneous problem
		\begin{align}\label{eq3111}
			\begin{cases}
				\partial^2_tu-\mathcal{L}u+\partial_tu-\mathcal{L}\partial_tu=F(t, x), & x\in G,t>0,\\
				u(0,x)=  u_0(x),  & x\in G,\\ \partial_tu(0, x)=  u_1(x), & x\in G,
			\end{cases}
		\end{align}
		can be expressed, by using Duhamel’s principle, as
		$$ u(t, x):= u_{0}(x)*_{(x)}E_{0}(t, x)+u_{1}(x)*_{(x)}E_{1}(t, x) +\int_{0}^{t} F(s, x)*_{(x)} E_{1}(t-s, x) \;d s,  $$
		where $*_{(x)}$ denotes the convolution with respect to the $x$ variable, $E_{0}(t, x)$ and $E_{1}(t, x)$  are  the fundamental solutions to the homogeneous problem (\ref{eq3111}), i.e., when  $F=0$ with initial data $\left(u_{0}, u_{1}\right)=\left(\delta_{0}, 0\right)$ and $\left(u_{0}, u_{1}\right)=$ $\left(0, \delta_{0}\right)$, respectively. 
		For any left-invariant differential operator $L$ on the compact Lie group $ {G}$, we applied  the property  that $L\left(v*_{(x)} E_{1}(t, \cdot)\right)=v *_{(x)} L\left(E_{1}(t, \cdot)\right)$ and   the invariance by time translations for the viscoelastic wave operator $	\partial^2_t-\mathcal{L}+\partial_t-\mathcal{L}\partial_t$ 
		in order to get the previous representation formula.
		
		\begin{defn}
			The function  $u$ is  said to be a mild solution to (\ref{eq3111})  on $[0, T]$ if $u$ is a fixed point for  the       integral operator  $N: u \in X(T) \rightarrow N u(t, x) $ defined as 
			\begin{align}\label{f2} 
				N u(t, x)= \varepsilon u_{0}(x) *_{(x)}  E_{0}(t, x)+\varepsilon u_{1}(x) *_{(x)}  E_{1}(t, x) +\int_{0}^{t}|u(s, x)|^{p} *_{(x)}  E_{1}(t-s, x) \;ds
			\end{align}
			in the evolution space $  \mathcal{C}^1 \left([0, T], H_{\mathcal{L}}^{1}(G)\right)  $, equipped with the norm defined in  (\ref{eq33333}). 
		\end{defn}

		As usual, the proof of the fact that, the map $N$ admits a uniquely determined fixed point for sufficiently small $T=T(\varepsilon),$  is based on Banach's fixed point theorem with respect to the norm on $X(T)$ as defined above.   More preciously, for $\left\|\left(u_{0}, u_{1}\right)\right\|_{H_{\mathcal{L}}^{1}(G) \times H_{\mathcal{L}}^{1}(G) }$   small enough,  if we can show the validity of the following two  inequalities
		$$\|N u\|_{X(T)} \leq C\left\|\left(u_{0}, u_{1}\right)\right\|_{H_{\mathcal{L}}^{1}(G) \times H_{\mathcal{L}}^{1}(G) }+C\|u\|_{X(T)}^{p},$$
		$$\|N u-N v\|_{X(T)} \leq C\|u-v\|_{X(T)}\left(\|u\|_{X(T)}^{p-1}+\|v\|_{X(T)}^{p-1}\right),$$
		for any $u, v \in X(T)$ and for some  suitable constant $C>0$ independent of $T$. Then by Banach's fixed point theorem  we can assure  that the operator $N$ admits a unique fixed point $u$.  This     function $u$ will be  the mild solution to (\ref{eq3111})  on $[0, T]$.  
		
		\vspace{2mm}
		
		In order to prove the local existence result, an important tool is the following Gagliardo-Nirenberg type inequality which can be derived from the general version of this inequality given in   \cite{Gall}. We also refer \cite{Gall} for the detailed proof of this inequality for more general connected unimodular Lie groups. 
		%	\begin{lem}\cite{Gall} \label{lemma1}
			%	Let $G$ be a connected unimodular Lie group with topological dimension $n.$ For any $1<q_0<\infty,~0<q,q_1<\infty$ and $0<\alpha <n$ such that $q_0<\frac{n}{\alpha},$ the following Gagliardo-Nirenberg type inequality holds
			%		\begin{align}\label{eq33}
				%		\|f\|_{L^q(G)}\lesssim %\|f\|^\theta_{H^{\alpha,q_0}_\mathcal L(G)}\|f\|^{1-\theta}_{L^{q_1}(G)}
				%	\end{align} 
			%	for all $f\in H^{1,q_0}_\mathcal L(G)\cap %L^{q_1}(G),$ provided that
			%	\begin{align*}
				%		\theta=\theta(n,\alpha, q,q_0,q_1)=\frac{\frac{1}{q_1}-\frac{1}{q}}{\frac{1}{q_1}-\frac{1}{q_0}+\frac{\alpha}{n}}\in[0,1].
				%	\end{align*}
			%	\end{lem} 
		
		\begin{lem} \label{lemma1}
			Let $G$ be a (connected) compact Lie group with topological dimension $n\geq3.$ Assume that  $q\ge2$ such that $q\leq\frac{2n}{n-2}$. Then the following Gagliardo-Nirenberg type inequality holds
			\begin{align}\label{eq34}
				\|f\|_{L^q(G)}\lesssim \|f\|^{\theta(n, q)}_{H^{1}_\mathcal L(G)}\|f\|^{1-\theta(n, q)}_{L^{2}(G)}
			\end{align} 
			for all $f\in H^{1}_\mathcal L(G)$, where $\theta(n, q)=n\left(\frac{1}{2}-\frac{1}{q} \right) $.
		\end{lem}
		One can also consult \cite{Gall, 27}  for several immediate important  remarks.

		\begin{proof}[Proof of Theorem \ref{thm22}]
			The  expression    (\ref{f2})  can be wriiten as  $N u=u^\sharp+I[u]$, where 
			\begin{align*}
				u^\sharp(t,x)=\varepsilon u_{0}(x) *_{(x)}  E_{0}(t, x)+\varepsilon u_{1}(x) *_{(x)}  E_{1}(t, x)
			\end{align*}
			and 
			\begin{align*}
				I[u](t,x):=\int\limits_0^t |u(s,x)|^p*_x E_1(t-s, x)ds.
			\end{align*} 
			
			Now for the part $u^\sharp$,	from  Theorem \ref{thm1},   immediately  it follows that
			\begin{align}\label{f3}
				\|u^\sharp\|_{X(T)}\lesssim\varepsilon\|(u_0,u_1)\|_{{H}_{\mathcal L}^1 (G)\times{H}_{\mathcal L}^1 (G)}.
			\end{align}
			On the other hand, for the part $I[u]$, using Minkowski's integral inequality, Young's convolution inequality, Gagliardo-Nirenberg type ineuqality \eqref{eq34}, Theorem \ref{thm1},  and  by time translation invariance property of the Cauchy problem (\ref{eq0010}), we get  
			\begin{align}\label{f}\nonumber
				\|\partial_t^j(-\mathcal L)^{i/2}I[u]\|_{L^2(G)}&=\left(\int_{G}	\big |\partial_t^j(-\mathcal L)^{i /2} \int\limits_0^t |u(s,x)|^p*_x E_1(t-s, x)ds\big |^2 dg\right)^{\frac{1}{2}}\\\nonumber
				&=\left(\int_{G}\big	|  \int\limits_0^t |u(s,x)|^p*_x \partial_t^j(-\mathcal L)^{i /2}E_1(t-s, x)ds\big|^2 dg\right)^{\frac{1}{2}}	\\\nonumber
				&\lesssim  \int\limits_0^t \| |u(s,\cdot )|^p*_x \partial_t^j(-\mathcal L)^{i /2}E_1(t-s, \cdot)\|_{L^2(G)}ds\\\nonumber
				&\lesssim  \int\limits_0^t \| u(s,\cdot)^p\|_{L^2(G)} \|\partial_t^j(-\mathcal L)^{i /2}E_1(t-s, \cdot)\|_{L^2(G)}ds\\\nonumber
				&\lesssim \int\limits_0^t  (1+t-s)^{-j-\frac{i}{2}} \|u(s,\cdot)\|^p_{L^{2p}(G)}ds\\\nonumber
				&\lesssim\int\limits_0^t  \|u(s,\cdot)\|^{p\theta(n,2p,  )}_{H^1_{\mathcal{L}}(G)}\|u(s,\cdot)\|^{p(1-\theta(n,2p ))}_{L^2(G)}ds \\ 
				&\lesssim t   \|u\|^p_{X(t)},
			\end{align} 
			for all $(i,j) \in\{(0, 0), (1, 0), (0,1), (1, 1)\}$.  Again for $(i,j) \in\{(0, 0), (1, 0), (0,1), (1, 1)\}$,  a similar calculations as in  (\ref{f}) togeter with Holder's inequality,  we get 
			\begin{align}\label{f5}\nonumber
				& 	\|\partial_t^j(-\mathcal L)^{i /2}\left(I[u]-I[v]\right)\|_{L^2(G)}\\\nonumber&\lesssim \int\limits_0^t (1+t-s)^{-j-\frac{i}{2}} \|u(s,\cdot)|^p-|v(s,\cdot)|^p\|_{L^{2}(G)}ds\\\nonumber
				&\lesssim\int\limits_0^t   \|u(s,\cdot)-v(s,\cdot)\|_{L^{2p}(G)}\left(\|u(s,\cdot)\|^{p-1}_{L^{2p}(G)}+\|v(s,\cdot)\|^{p-1}_{L^{2p}(G)}\right)ds\\ 
				&\lesssim t  \|u-v\|_{X(t)}\left(\|u\|^{p-1}_{X(t)}-\|v\|^{p-1}_{X(t)}\right).
			\end{align} 
			Thus 	combining  (\ref{f3}),   (\ref{f}), and (\ref{f5}), we have
			\begin{align}\label{1}
				\|N u\|_{X(t)} \leq D \varepsilon\left\|\left(u_{0}, u_{1}\right)\right\|_{H_{\mathcal{L}}^{1 }(G) \times H_{\mathcal{L}}^{1 }(G)}+DT\|u\|_{X(t)}^{p} 
			\end{align} 
			and 
			\begin{align}\label{2} \|Nu-Nv\|_{X(T)}\leq  DT \|u-v\|_{X(t)}\left(\|u\|^{p-1}_{X(T)}-\|v\|^{p-1}_{X(T)}\right).\end{align} 
			
			Thus for   sufficiently small   $T$,  the map $N$ turns out to be a contraction in some neighbourhood of $0$ in the Banach space $X(T).$ Therefore, it follows from the Banach's fixed point theorem that there exists a uniquely determined fixed point $u$ of the map $N.$ This fixed point $u$ is  the mild solution to the system (\ref{eq0010}) on $[0, t] \subset[0, T]$.    
		\end{proof}

		\section*{Acknowledgement}
		Arun Kumar Bhardwaj  thanks  IIT Guwahati for providing financial support.  He also thanks  his supervisor  Rajesh Srivastava  for his support and encouragement. Vishvesh Kumar is supported  by the FWO  Odysseus  1  grant  G.0H94.18N:  Analysis  and  Partial Differential Equations and  Partial Differential Equations and by the Methusalem programme of the Ghent University Special Research Fund (BOF) 	(Grant number 01M01021).   Shyam Swarup Mondal thanks IIT Delhi  for providing financial support. 
		\section{Data availability statement}
		The authors confirm that the data supporting the findings of this study are available within the article  and its supplementary materials.

	\end{document}